\documentclass[11pt, english]{article}

\usepackage[english]{babel}
\usepackage[thms]{packages_old}
\usepackage[hidelinks]{hyperref}
\usepackage[capitalise]{cleveref}
\usepackage{todonotes}
\usepackage{dsfont}
\usepackage{orcidlink}

\usepackage[style=trad-alpha,backend=biber,url=false, isbn=false]{biblatex}
\newcommand{\dy}{\;\mathrm{d}}

\title{Weak reservoirs are superexponentially irrelevant for misanthrope processes}
\author{Julian Kern\orcidlink{0000-0002-8231-0736}}

\begin{document}

\maketitle
\begin{abstract}
We provide a short proof for the exponential equivalence between misanthrope processes in contact with weak reservoirs and those with impermeable boundaries.
As a consequence, we can derive both the hydrodynamic limit and the large deviations of the totally asymmetric simple exclusion process (TASEP) in contact with weak reservoirs.
This extends a recent result which proved the hydrodynamic behaviour of a vanishing viscocity approximation of the TASEP in contact with weak reservoirs.
Further applications to a class of asymmetric exclusion processes with long jumps are discussed.
\end{abstract}

\section{Introduction}

When new particle systems are introduced, they are often first considered on the whole lattice $\mathbb{Z}^d$ or large tori $(\mathbb{Z}/N\mathbb{Z})^d$.
When restricting the system to a bounded box and adding boundary interactions, additional difficulties come into play, both from the microscopic and the macroscopic level.
Although progress has been made, many non-reversible asymmetric boundary-driven particle systems still remain out of reach of current methods.
This includes even some of the simplest models such as the Totally Asymmetric Simple Exclusion Process (TASEP), for which the boundary-driven dynamics have been unsolved so far, see e.g.~\cite{X22a,X22b}.

The aim of this paper is to provide a new tool to analyze specifically processes in contact with \emph{weak reservoirs}, that is, particle systems in which the boundary dynamics act on a slower time scale than the bulk dynamics.
More precisely, we provide a short proof of the exponential equivalence between processes in contact with weak reservoirs, respectively with impermeable boundaries, for a large class of models.
The result is built on a comparison argument that relies heavily on the coupling argument introduced in \cite{Coc85} for misanthrope processes.
In the special case of one-dimensional totally asymmetric processes, this comparison allows us to deduce the limiting behaviour from the corresponding particle system on the whole line.
As an immediate consequence, we can extend the result from \cite{X22b} to the original TASEP (without the need of vanishing viscocity dynamics) and derive the hydrodynamic behaviour of related totally asymmetric models.

The remainder of the paper is divided into three sections: the presentation of the main result, its applications, and the proof. A generalization is discussed in \Cref{asec:misanthrope}. 
\Cref{asec:extension_proof} contains a technical extension of the hydrodynamic limit proved in \cite{SS18}, necessary for the application of the main result to a totally asymmetric exclusion process with long jumps.

\paragraph*{Acknowledgement}

This research has been funded by the Deutsche Forschungsgemeinschaft (DFG, German Research Foundation) under Germany’s Excellence Strategy – The Berlin Mathematics Research Center MATH+ (EXC-2046/1, project ID: 390685689) and through grant CRC 1114 Scaling Cascades in Complex Systems, Project
Number 235221301, Project C02 Interface dynamics: Bridging stochastic and hydrodynamic descriptions.
I would like to thank Robert Patterson for valuable discussions.

\section{Notation and main result}

In this section, we will restrict ourselves to the special case of \emph{exclusion} processes. For the more general case, see \Cref{asec:misanthrope}.

Write $\Lambda_N:=\{1,\dots, N-1\}$ for the bulk and $\Omega_N := \{0,1\}^{\Lambda_N}$ for the space of configurations. For a configuration $\eta\in\Omega_N$, the occupation variable $\eta(x)$ records whether there is a particle at site $x$ or not. We denote by $\eta^{x,y}$ the exchange of the sites $x$ and $y$, and by $\eta^x$ the flip of the site $x$ in the sense that
\[
\eta^{x,y}(z) := \begin{cases}
\eta(y) & \text{ if }z = x\\
\eta(x) & \text{ if }z = y\\
\eta(z) & \text{ otherwise }
\end{cases}\qquad\text{ and }\qquad \eta^x(z) := \begin{cases}
1 - \eta(x) & \text{ if }z = x\\
\eta(z) & \text{ otherwise }
\end{cases}.
\]

Next, define the generators
\[
\begin{array}{rl}
\mathcal{L}_{\text{bulk}}^N f(\eta) &:=\displaystyle \sum_{x,y\in\Lambda_N} p(x,y)\cdot \eta(x)\big(1-\eta(y)\big)\cdot\Big( f(\eta^{x,y}) - f(\eta)\Big),\\
\mathcal{L}_{\text{influx}}^N f(\eta) &:=\displaystyle \sum_{x\not\in\Lambda_N}\sum_{y\in\Lambda_N} p(x,y) \cdot \alpha(x)\big(1-\eta(y)\big)\cdot\Big( f(\eta^y) - f(\eta)\Big),\\
\mathcal{L}_{\text{outflux}}^N f(\eta) &:=\displaystyle \sum_{x\in\Lambda_N}\sum_{y\not\in\Lambda_N} p(x,y) \cdot \eta(x)\beta(y)\cdot \Big( f(\eta^x) - f(\eta)\Big),
\end{array}
\]
where $p:\mathbb{Z}^2 \rightarrow [0,+\infty)$ is a jump kernel and $\alpha,\beta:\mathbb{Z}\rightarrow [0,+\infty)$ are bounded. Here and in the following, we will use the shortcut $x\not\in\Lambda_N$ to mean $x\in \mathbb{Z}\setminus \Lambda_N$.

For two sequences $(a_N)_{N\in\mathbb{N}}, (b_N)_{N\in\mathbb{N}}$ of positive real numbers, we will write 
\[
a_N = \begin{cases}
o(b_N) & \text{ if }\lim_{N\to\infty}\frac{a_N}{b_N} = 0\\
\mathcal{O}(b_N) & \text{ if } \limsup_{N\to\infty} \frac{a_N}{b_N} < +\infty\\
\Theta(b_N) & \text{ if } a_N = \mathcal{O}(b_N)\text{ and } b_N = \mathcal{O}(a_N)
\end{cases}.
\]

\begin{defin}
\begin{enumerate}
	\item An \emph{exclusion process in contact with impermeable boundaries} is an $\Omega_N$-valued Markov process with generator $\mathcal{L}_{\textnormal{bulk}}^N$.
	\item An \emph{exclusion process in contact with weak reservoirs} is an $\Omega_N$-valued Markov process with generator $\mathcal{L}^N := \mathcal{L}_{\textnormal{bulk}}^N + \theta(N)\left(\mathcal{L}_{\textnormal{influx}}^N + \mathcal{L}_{\textnormal{outflux}}^N\right)$ for some $\theta(N) = o(1)$.
\end{enumerate}
\end{defin}

In the following, we will identify an $\Omega_N$-valued process with the corresponding measure-valued process via the map
\begin{equation}\label{eq:empirical_map}
\pi^N:\Omega_N \rightarrow \mathcal{M}_F([0,1]),\qquad \eta\mapsto\dfrac{1}{N-1}\sum_{x\in\Lambda_N}\eta(x)\delta_{x/N},
\end{equation}
where $\mathcal{M}_F([0,1])$ denotes the space of finite measures on $[0,1]$.
We endow $\mathcal{M}_F([0,1])$ with the topology of weak convergence, induced by the Lévy-Prokhorov metric
\[
d_{LP}(\mu, \nu) := \inf\{ \epsilon \textgreater 0\;:\; \forall A\subseteq[0,1]\text{ Borel }, \mu(A) \leq \nu(A_\epsilon) + \epsilon\},
\]
where $A_\epsilon := \{ y\in [0,1]\;:\; \inf_{x\in A} \vert y-x\vert \leq \epsilon\}$.
Next, consider the space of measure-valued càdlàg processes $\mathbb{D}_{[0,T]}\big(\mathcal{M}_F([0,1])\big)$ with the Skorokhod topology induced by the weak topology on $\mathcal{M}_F([0,1])$.
In the following, we will write $d_{J_1}$ for the corresponding complete metric inducing this topology.

Recall that the total variation norm of a finite measure is given by
\[
\Vert \mu\Vert_{TV} := \sup_{A\subseteq [0,1]\text{ Borel}} \vert \mu(A)\vert.
\]
From the definition of the Lévy-Prokhorov metric, it follows that
\[
d_{LP}(\mu, \nu) \leq \Vert \mu - \nu\Vert_{TV},
\]
so that we may strengthen the Skorokhod topology by replacing it with the uniform topology in total variation induced by
\[
\Delta(\pi_1,\pi_2) := \sup_{t\in [0,T]} \Vert \pi_1(t) - \pi_2(t)\Vert_{TV} \geq d_{J_1}(\pi_1,\pi_2)
\]

\begin{defin}[Exponential equivalence, cf.~{\cite[Definition 4.2.10]{DZ98}}]
We say that two sequences of probability measures $\mathbb{Q}_1^N$ and $\mathbb{Q}_2^N$ on $\mathbb{D}_{[0,T]}\big(\mathcal{M}_F([0,1])\big)$ are \emph{exponentially equivalent} if there exist
\begin{enumerate}[label=\roman*)]
	\item a sequence of (abstract) probability spaces $(\mathcal{X}^N, \mathcal{F}^N, \mathcal{Q}^N)$ and
	\item a sequence of random variables $\pi_1^N,\pi_2^N:\mathcal{X}^N \rightarrow \mathbb{D}_{[0,T]}\big(\mathcal{M}_F([0,1])\big)$ with respective laws $\mathbb{Q}_1^N$ and $\mathbb{Q}_2^N$
\end{enumerate}
such that the sets $\{d_{J_1}(\pi_1^N, \pi_2^N) > \epsilon\}$ are $\mathcal{F}^N$-measurable and
\[
\limsup_N \dfrac{1}{N}\ln\mathcal{Q}^N\left\{ d_{J_1}(\pi_1^N, \pi_2^N) > \epsilon\right\} = -\infty
\]
for all $\epsilon > 0$.
\end{defin}

In order to state the main result, we introduce the notion of a time change. 
We say that we \emph{speed up} a process $(\pi_t)_{t\geq 0}$ by a factor $\kappa$ if we consider the process $(\pi_{\kappa t})_{t\geq 0}$ instead.
In the case of a Markov process with generator $\mathcal{L}$, this is equivalent to considering the Markov process with generator $\kappa\mathcal{L}$.

\begin{theo}\label{theo:expon_equiv}
If both processes are sped up by some factor $\kappa(N)$, then the exclusion process in contact with impermeable boundaries is exponentially equivalent to the exclusion process in contact with weak boundaries provided
\[
\kappa(N)\theta(N)\sum_{x\in\Lambda_N,y\not\in\Lambda_N} \Big( p(x,y) + p(y,x)\Big) = o(N).
\]
\end{theo}

Under the assumptions of the theorem, the two processes are indistinguishable up to the level of large deviations.
In particular, their hydrodynamic behaviour coincides and if one satisfies a law of large numbers, the other satisfies the same.

\section{Applications to totally asymmetric systems}\label{sec:applications}

In this section, we discuss applications of the exponential equivalence to prove the hydrodynamic behaviour and the large deviations for asymmetric exclusion processes.

\subsection{The TASEP in contact with weak reservoirs}

In the recent paper \cite{X22b}, one of the models is the nearest-neighbour TASEP in contact with weak reservoirs given by the above through the choice $p(x,y) = \mathds{1}_{y=x+1}$, $\alpha$ and $\beta$ some (possibly time-dependent, but locally) bounded rates, and $\theta(N) = N^m$ for some $m < 0$. 
Due to the asymmetry, the process evolves on the timescale $\kappa(N) = N$.
This means in particular that
\[
\kappa(N)\theta(N)\sum_{x\in\Lambda_N,y\not\in\Lambda_N} \Big( p(x,y) + p(y,x)\Big) = \mathcal{O}\big(N\theta(N)\big) = o(N),
\]
verifying the condition of \Cref{theo:expon_equiv}. To include time-dependent rates, we make use of the more general result discussed at the end of \Cref{asec:misanthrope}.

\Cref{theo:expon_equiv} then ensures that the process is exponentially equivalent to the (nearest-neighbour) TASEP with impermeable boundaries.
We can extend the latter to the left by zeros and to the right by ones without changing the dynamics, transforming it into the TASEP on $\mathbb{Z}$.
Using \cite{Sep98} (see also \cite[Theorem 2.1]{Var04}), we obtain the hydrodynamic behaviour.
The large deviations are considered in \cite{Var04} and completed in \cite[Main Theorem]{QT22}.
Through the exponential equivalence, the LDP translates directly to the TASEP in contact with weak reservoirs.

In contrast to \cite{X22b}, this result does not necessitate the addition of a vanishing viscocity to the model in order to prove the hydrodynamic behaviour.

\subsection{The TALJEP in contact with weak reservoirs}

Similarly, we may consider the jump kernel
\[
p(x,y) = \dfrac{\mathds{1}_{y > x}}{\vert y - x\vert^{1+\gamma}}
\]
for some $\gamma > 0$.
As a process on $\mathbb{Z}$, this is a special case of the model from \cite{SS18}.
The Totally Asymmetric Long Jump Exclusion Process (TALJEP) in contact with reservoirs with $\gamma \in (0,1)$ is also considered as a special case in \cite{KGX24}, where the hydrodynamic behaviour is investigated through other methods.

The TALJEP on the whole of $\mathbb{Z}$ undergoes a phase transition at $\gamma = 1$: for $\gamma\in (0,1)$, the mean jump size is infinite and the long range effects remain visible at the macroscopic level; for $\gamma > 1$, the mean becomes finite and the model behaves like the TASEP.
The phase transition can also be read off from the correct time scales given by
\[
\kappa(N) = \begin{cases}
N^\gamma & \text{ if }\gamma < 1\\
\frac{N}{\ln N} & \text{ if } \gamma = 1\\
N & \text{ if } \gamma > 1
\end{cases},
\]
see \cite{SS18} for details.

In the case of weak boundaries, we may check that for every $\gamma > 0$ and the above choices of $\kappa(N)$, one has
\[
\kappa(N)\theta(N)\sum_{x\in\Lambda_N, y\not\in\Lambda_N} \Big( p(x,y) + p(y,x)\Big) = o(N),
\]
so that \Cref{theo:expon_equiv} is applicable.
As in the case of the TASEP, we can extend the TALJEP with impermeable boundaries to the left with zeros and the right with ones to recover the TALJEP on $\mathbb{Z}$.
However, this does not yet allow us to deduce the hydrodynamical behaviour of the TALJEP with weak boundaries.
Indeed, the proof in \cite{SS18} relies on the assumption that the initial profile has the same asymptotic density $\rho^*\in(0,1)$ in both directions.
This is violated here in multiple ways, since the asymptotic density to the left is 0 and the asymptotic density to the right is 1.

In the fractional regime $\gamma\in (0,1)$, the proof from \cite{SS18} does actually work for general initial conditions; for a different approach which extends the result to measurable initial conditions, see \cite{KGX24}.
In the hyperbolic regime $\gamma > 1$, we will control the propagation of mass ´to extend the result from \cite{SS18} to suitable initial profiles.
This allows us to deduce the hydrodynamic behaviour of the TALJEP in contact with weak reservoirs at least when a) the initial profile $\rho_0$ is continuous on $[0,1]$ and satisfies $\rho_0(0) = 0$ and $\rho_0(1) = 1$, and b) the initial configuration is distributed as a product measure with profile $\rho_0$, as it is given by the next result.

\begin{lem}\label{lem:compactly_supported}
Let $\rho_0\in C(\mathbb{R};[0,1])$ be a continuous profile satisfying $\rho_0(x) = 0$ (resp. {$\rho_0(x) = 1$}) for $x$ small (resp.~large) enough, and let $\mu_0^N$ be the product measure on $\mathbb{Z}$ with marginals $\mu_0^N\big(\eta(x) = 1\big) = \rho_0(x/N)$.
Then, the TALJEP on $\mathbb{Z}$ for $\gamma > 1$ with initial configuration distributed as $\mu_0^N$ satisfies a law of large numbers with hydrodynamic limit given by the solution $\rho$ to Burgers' equation \cite[Equation (3.2)]{SS18} with initial value $\rho_0$.
\end{lem}
\begin{proof}
See \Cref{asec:extension_proof} after reading Section 4.
\end{proof}

\section{Proof of \Cref{theo:expon_equiv}}\label{sec:proof}

For every $N\in\mathbb{N}$, let $\mu_N$ be a probability measure on $\Omega_N$.
For two configurations ${\eta_1,\eta_2\in\Omega_N}$, we will say that $\eta_1 \leq \eta_2$ if the configuration are ordered pointwisely, i.e.~if $\eta_1(x) \leq \eta_2(x)$ for all $x\in \Lambda_N$.
At the end of this section, we will construct a coupling of the three processes $\eta^N$, $\tilde{\eta}^N$ and $\hat{\eta}^N$ with the same initial distributions and respective generators
\[
\kappa(N)\mathcal{L}_{\text{bulk}}^N,\quad \kappa(N)\mathcal{L}_{\textnormal{bulk}}^N + \theta(N)\kappa(N)\mathcal{L}_{\text{influx}}^N\quad\text{and}\quad \kappa(N)\mathcal{L}^N
\]
that additionally satisfy
\[
\eta^N_t \leq \tilde{\eta}^N_t\qquad\text{ and }\qquad \hat{\eta}^N_t \leq \tilde{\eta}^N_t.
\]
Note that $\eta^N$ corresponds to the process with impermeable boundaries and $\hat{\eta}^N$ is the process with weak boundaries.


Suppose for a moment that we are given such a coupling. 
Write $\pi^N_t := \pi^N(\eta^N_t)$, $\tilde{\pi}^N_t := \pi^N(\tilde{\eta}^N_t)$ and $\hat{\pi}^N_t := \pi^N(\hat{\eta}^N_t)$ for the corresponding measure-valued processes.
Since they are all atomic, the ordering implies
\[
N\Vert \tilde{\pi}^N_t - \pi^N_t\Vert_{TV} = \vert \tilde{\eta}^N_t\vert - \vert \eta^N_t\vert\quad\text{ and }\quad N\Vert \tilde{\pi}^N_t - \hat{\pi}^N_t\Vert_{TV} = \vert \tilde{\eta}^N_t \vert - \vert \hat{\eta}^N_t\vert,
\]
where we write $\vert \eta\vert := \sum_{x\in\Lambda_N} \eta(x)$ for the total mass of a configuration $\eta\in\Omega_N$.
Furthermore, since the difference in mass can only come from the boundary interactions, we may conclude that both $t\mapsto \vert \tilde{\eta}^N_t\vert - \vert \eta_t^N\vert$ and $t\mapsto \vert \tilde{\eta}^N_t \vert - \vert \hat{\eta}_t^N\vert$ are non-decreasing in time, see also the coupling at the end of this section.

Using the comparison between the Skorokhod metric w.r.t.~weak convergence and the uniform metric w.r.t.~total variation, we have
\[
\left\{ d_{J_1}(\tilde{\pi}^N, {\pi}^N) > \epsilon\right\} \subseteq \left\{ \sup_{t\in [0,T]} \vert \tilde{\eta}_t^N \vert - \vert \eta_t^N\vert > \epsilon N\right\} = \left\{ \vert \tilde{\eta}_T^N\vert - \vert \eta_T^N\vert > \epsilon N\right\}
\]
and similarly for $\{d_{J_1}(\tilde{\pi}^N, \hat{\pi}^N) > \epsilon\}$.
Hence, the proof reduces to the following lemma.

\begin{lem}[No loss of mass]\label{lem:loss_mass}
The events $\{ \vert \tilde{\eta}^N_T \vert - \vert \eta^N_T\vert \geq \epsilon N\}$ and $\{\vert \tilde{\eta}_T^N \vert - \vert \hat{\eta}^N_T\vert \geq \epsilon N\}$ are superexponentially unlikely for every $\epsilon > 0$.
\end{lem}
\begin{proof}
As both events are analogous, we will concentrate on the first set only.
Since the change in mass can come only from the influx of particles, it suffices to prove that the probability of $\epsilon N$ particles entering up to time $T$ is superexponentially small.
Note that the number of particles is bounded from above by a Poisson number $P$ with parameter 
\[
\lambda_N := T\Vert \alpha\Vert_\infty\cdot\theta(N)\kappa(N)\sum_{x\in\Lambda_N, y\not\in\Lambda_N} p(y,x) = o(N).
\]
The usual Chernoff bound provides us with the estimate
\[
 	\mathbb{P}(P \geq \lambda_N + x) \leq \exp\left(- \dfrac{x^2}{ \lambda_N } \cdot h\left(\dfrac{x}{\lambda_N}\right)\right),
 	\]
 	where $h(u) = \frac{(1+u)\ln(1+u) - u}{u^2}$ vanishes at infinity like $\frac{\ln u}{u}$, see {e.g.}~\cite{Can19}.
In particular, we conclude that
\begin{align*}
\mathbb{P}\left(\exists t\in [0,T]\;:\; \vert \tilde{\eta}_t^N\vert - \vert \eta_t^N\vert \geq \epsilon N\right) \leq  \exp\left( - \Theta\left(\epsilon N \cdot \ln \left(\dfrac{\epsilon N}{o(N)}\right)\right)\right)
\end{align*}
which is superexponentially small.
\end{proof}

The remainder of this section is dedicated to presenting the core idea in the construction of the coupling used in the proof.
It is an adaptation of the misanthrope (or: attractive) coupling introduced in \cite{Coc85}, see also \cite[Section 9]{SS18} for an English version.
In the following, we will only treat the coupling $(\eta^N, \tilde{\eta}^N)$ as the other coupling is analogous.

Consider the generator
\begin{align*}
&\quad\;\overline{\mathcal{L}}^Nf(\xi,\zeta) \\
&= \kappa(N)\sum_{x,y\in\Lambda_N} p(x,y)\cdot \Big( \xi(x)\big(1 - \xi(y)\big) \wedge \zeta(x)\big(1- \zeta(y)\big)\Big) \cdot\Big( f(\xi^{x,y},\zeta^{x,y}) - f(\xi,\zeta)\Big)\\
&\quad + \kappa(N)\sum_{x,y\in \Lambda_N} p(x,y)\cdot \Big( \xi(x)\big(1 - \xi(y)\big) - \xi(x)\big(1 - \xi(y)\big) \wedge \zeta(x)\big(1- \zeta(y)\big)\Big) \\
&\hspace*{4cm}\cdot\Big( f(\xi^{x,y},\zeta) - f(\xi,\zeta)\Big)\\
&\quad + \kappa(N)\sum_{x,y\in \Lambda_N} p(x,y)\cdot \Big( \zeta(x)\big(1 - \zeta(y)\big) - \xi(x)\big(1 - \xi(y)\big) \wedge \zeta(x)\big(1- \zeta(y)\big)\Big) \\
&\hspace*{4cm}\cdot\Big( f(\xi,\zeta^{x,y}) - f(\xi,\zeta)\Big)\\
&\quad + \theta(N)\kappa(N)\sum_{x\not\in \Lambda_N}\sum_{y\in\Lambda_N} p(x,y)\alpha(x)\big(1 - \zeta(y)\big)\cdot \Big( f(\xi, \zeta^y) - f(\xi,\zeta)\Big),
\end{align*}
so that, whenever possible, particles jump together.
If we consider the corresponding Markov process $(\xi^N, \zeta^N)_N$, one clearly has that $\xi^N$ is a Markov process with generator $\kappa(N)\mathcal{L}^N_{bulk}$ and $\zeta^N$ is a Markov process with generator $\kappa(N)\mathcal{L}^N_{bulk} + \theta(N)\kappa(N)\mathcal{L}^N_{influx}$.

The test function $\mathds{1}_{\xi\leq \zeta}$ yields that
\begin{align*}
t\mapsto \mathds{1}_{\xi^N_t\leq \zeta^N_t} - \int_0^t \overline{\mathcal{L}}^N \mathds{1}_{\xi^N_s\leq \zeta^N_s}\;\mathrm{d}s
\end{align*}
is a martingale w.r.t.~its natural filtration.
Next, define the stopping time 
\[
\tau := \inf\{t\geq 0\;:\; \xi^N_t > \zeta^N_t\}.
\]
Since
\[
\mathbb{E}\left[\mathds{1}_{\xi^N_{t\wedge\tau}\leq \zeta^N_{t\wedge\tau}}\right] = 1 - \mathbb{P}(\tau \leq t),
\]
we obtain through the martingale property that
\begin{align*}
\mathbb{P}(\tau \leq t) &= \mathbb{P}(\xi^N_0 \textgreater \zeta^N_0) + \mathbb{E}\left[\int_0^{t\wedge\tau} \overline{\mathcal{L}}^N \mathds{1}_{\xi^N_s \leq \zeta^N_s}\;\mathrm{d}s\right].
\end{align*}
Since the integrand is zero on the set $\{t \textless \tau\}$, we conclude that if $\xi_0^N \leq \zeta_0^N$ a.s., then $\xi_t^N \leq \zeta_t^N$ for every $t\in [0,T]$ a.s.

\appendix

\section{Generalization to misanthrope processes}\label{asec:misanthrope}

The proof in \Cref{sec:proof} relies heavily on the attractive coupling of the different processes.
From \cite{Coc85}, it is known that this type of coupling can be constructed for the large class of \emph{misanthrope processes}.
These cover many models of interest, including the exclusion and zero range processes.

Let $k\in \mathbb{N}\cup \{\infty\}$ denote the maximal number of particles allowed at a site and set $S_k := \{0, \dots, k\}$ or $S_\infty := \mathbb{N}_0$ accordingly.

For $N\in\mathbb{N}$, $N\geq 2$, define the \emph{bulk} $\Lambda_N := \{1,\dots, N-1\}$ and the space of configurations $\Omega_N := S_k^{\Lambda_N}$. 
For a configuration $\eta\in\Omega_N$ and sites $x,y\in\Lambda_N$, we define the three actions $\eta\mapsto\eta^{x\rightarrow y}$, $\eta\mapsto \eta^{x\uparrow}$ and $\eta\mapsto \eta^{x\downarrow}$ as follows:
\begin{enumerate}
	\item if $\eta(x) = 0$ or $\eta(y) = k$, set $\eta^{x\rightarrow y} := \eta$, otherwise set $\eta^{x\rightarrow y} := \eta - \delta_x + \delta_y$, {i.e.}
	\[
	\eta^{x\rightarrow y}(z) := \begin{cases}
		\eta(x) - 1 & \text{ if }z = x\\
		\eta(y) + 1 & \text{ if }z = y\\
		\eta(z) & \text{ otherwise}
	\end{cases};
	\]
	\item if $\eta(x) = k$, set $\eta^{x\uparrow} := \eta$, otherwise set $\eta^{x\uparrow} := \eta + \delta_x$;
	\item if $\eta(x) = 0$, set $\eta^{x\downarrow} := \eta$, otherwise set $\eta^{x\downarrow} := \eta - \delta_x$.
\end{enumerate}
For functions $f:\Omega_N\rightarrow\mathbb{R}$, define the generators
\[
\mathcal{L}_{\text{bulk}}^N f(\eta) := \sum_{x,y\in\Lambda_N}^{N-1} p(x,y)\cdot b_{\text{bulk}}\big(\eta(x), \eta(y)\big)\cdot \Big( f(\eta^{x,y}) - f(\eta)\Big)
\]
and
\begin{align*}
\mathcal{L}_{\text{influx}}^N f(\eta) &:= \sum_{x\not\in\Lambda_N}\sum_{y\in\Lambda_N} p(x,y)\cdot b_{\text{influx}}\big(x, \eta(y)\big)\cdot\Big( f(\eta^{y\uparrow})- f(\eta)\Big),\\
\mathcal{L}_{\text{outflux}}^N f(\eta) &:= \sum_{x\in\Lambda_N} \sum_{y\not\in\Lambda_N} p(x,y)\cdot b_{\text{outflux}}\big(\eta(x), y\big) \cdot \Big( f(\eta^{x\downarrow}) - f(\eta)\Big),
\end{align*}
where
\begin{enumerate}[label=\roman*)]
	\item $p$ is a jump kernel,
	\item $b_{\text{bulk}}:S_k^2\rightarrow [0,+\infty)$ is non-decreasing in its first variable, non-increasing in its second variable and  satisfies $b(n,m) = 0$ if and only if $n = 0$ or $m=k$,
	\item $b_{\text{influx}}:\mathbb{Z}\times S_k\rightarrow [0,+\infty)$ and $b_{\text{outflux}}:S_k\times \mathbb{Z}\rightarrow [0,+\infty)$ are bounded.
\end{enumerate} 

We will assume that $b_{\text{bulk}}$ is such that the following Markov processes exist. This is trivially satisfied when $b_{\text{bulk}}$ is bounded, {e.g.}~if $k \neq \infty$.

\begin{defin}
A \emph{misanthrope process in contact with impermeable boundaries} is defined through the generator $\mathcal{L}_{\textnormal{bulk}}^N$, whereas a \emph{misanthrope process in contact with weak reservoirs} has the generator
\[
\mathcal{L}^N := \mathcal{L}_{\textnormal{bulk}}^N + \theta(N)\left( \mathcal{L}_{\textnormal{influx}}^N + \mathcal{L}_{\textnormal{outflux}}^N\right)
\]
for some $\theta(N) = o(1)$.
\end{defin}

As before, we will identify processes with values in the space of configurations $\Omega_N$ with the corresponding process with values in the space of measures via the map $\pi^N$, cf.~\Cref{eq:empirical_map}.

\begin{theo}\label{theo:general_exp_equiv}
If both processes are sped up by a factor $\kappa(N)$, then the misanthrope process in contact with weak reservoirs is exponentially equivalent to the misanthrope process in contact with impermeable boundaries provided
\[
\kappa(N)\theta(N)\sum_{x\in\Lambda_N, y\not\in\Lambda_N} \Big(p(x,y) + p(y,x)\Big) = o(N).
\]
\end{theo}
\begin{proof}
The proof is exactly as in \Cref{sec:proof}. 
The only difference is that we replace $\Vert \alpha\Vert_\infty$ by $\Vert b_{\text{influx}}\Vert_\infty$ and similarly for $\beta$ and $b_{\text{outflux}}$.
\end{proof}

Note that the proof does not depend on the underlying space. 
In particular, the result can be extended to misanthrope processes on any lattice, including $\mathbb{Z}^d$.
Furthermore, it can be generalized to time-dependent interactions with the reservoirs as long as they are locally $L^1$ in time in the sense that $b_{\text{influx}}\in L^1_{loc}([0,+\infty); L^\infty( \mathbb{Z}\times S_k))$ and similarly for $b_{\text{outflux}}$.
In this case $T\Vert b_{\text{influx}}\Vert_\infty$ is to be replaced by $\int_0^T \Vert b_{\text{influx}}(t)\Vert_\infty \dy t$.

Although pathological counter-examples can be constructed, the statement of \Cref{theo:general_exp_equiv} is sharp in most situations.
This includes also the symmetric case, see e.g.~\cite{BMNS17,BCGS22} for the treatment of the symmetric exclusion process with nearest-neighbour and long jumps, respectively, in contact with weak reservoirs.
It is generally equally hard to derive the hydrodynamic behaviour of the process in contact with weak reservoirs or in contact with impermeable boundaries, so that the result only slightly shortens proofs by providing a general argument for why boundary terms may be ignored.

In the context of totally asymmetric processes, however, \Cref{theo:general_exp_equiv} provides a shortcut for proving the hydrodynamic behaviour (and even higher order behaviour as the fluctuations or the large deviations) as shown in \Cref{sec:applications}.

\section{Proof of \Cref{lem:compactly_supported}}\label{asec:extension_proof}

For simplicity, assume that $\rho_0(x) = 0$ for $x \leq 0$ and $\rho_0(x) = 1$ for $x\geq 1$.
Let $(\rho_0^{m,\ell})_{m,\ell\in\mathbb{N}}$ be a family of continuous functions satisfying $\rho_0^{m,\ell}\vert_{[-m,\ell]} = \rho_0\vert_{[-m,\ell]}$ and $\rho_0^{m,\ell}(x) = \frac{1}{2}$ on $\mathbb{R}\setminus [-2m,2\ell]$.
We may choose the family such that it is pointwise non increasing in $m$ and pointwise non decreasing in $\ell$.
Write $\mu_0^{N, m,\ell}$ for the corresponding product measure on $\{0,1\}^\mathbb{Z}$. Furthermore, denote by $\mu_0^{N, \ell}$ the measures obtained from the pointwise limit $\lim_m \rho_0^{m, \ell}$ which vanishes to the left of $0$.

Using the attractive coupling, we may construct the TALJEPs $\eta^{N, m,\ell}$ and $\eta^{N, \ell}$ on $\mathbb{Z}$ started from $\mu_0^{N,m,\ell}$ and $\mu_0^{N,\ell}$ on a common probability space such that
\[
\eta_t^{N, \ell} \leq \eta_t^{N, m, \ell} \leq \eta_t^{N, m', \ell}\qquad\text{ and }\qquad \eta_t^{N,\ell} \leq \eta_t^{N, \ell'} \leq \eta_t^N,
\] 
for any $m' \leq m$, $\ell \leq \ell'$ and $t\geq 0$, a.s.

Similarly to the proof of \Cref{theo:expon_equiv}, it is enough to show that both
\[
\lim_{m}\sup_\ell \limsup_N \mathbb{P}^N\left( \sum_{x\geq 1} \eta_T^{N, m, \ell}(x) - \eta_T^{N,\ell}(x) > \epsilon N\right) = 0
\]
and
\[
\lim_\ell \limsup_N \mathbb{P}^N\left( \sum_{x=1}^{N-1} \eta_T^N(x) - \eta_T^{N,\ell}(x) > \epsilon N\right) = 0.
\]

As both quantities are similar, we will concentrate on the former.
By construction, $\sum_{x\geq 1} \eta_t^{N,\ell}(x) = \sum_{x\geq 1} \eta_0^{N,\ell}(x) = \sum_{x\geq 1}\eta_0^{N,m,\ell}(x)$.
In particular, it is enough to bound
\begin{equation}\label{eq:final_equation}
\mathbb{P}^N\left( \sum_{x\geq 1} \eta_T^{N,m,\ell}(x) - \eta_0^{N,m,\ell}(x) > \epsilon N\right)\leq \dfrac{1}{\epsilon N}\mathbb{E}^N\left[\sum_{x\geq 1} \eta_T^{N,m,\ell}(x) - \eta_0^{N,m,\ell}(x)\right].
\end{equation}
Since the exclusion restraint slows down the motion of particles, we can bound this quantity by studying the dynamics of independent random walkers with jump kernel $p$.
Write $X_t = \sum_{i=1}^{{\mathcal{N}}_t} \xi_i$ for a random walker with jump kernel $p$ starting at the origin, where ${\mathcal{N}}$ is a Poisson process with rate given by the total jump rate and the $\xi_i$ are iid drawn from the jump distribution, independently of ${\mathcal{N}}$.
Then, for $\gamma' := \frac{\gamma - 1}{2} > 0$, Jensen's inequality yields
\begin{align*}
\mathbb{E}\left[ X_T^{1+\gamma'}\right] &= \mathbb{E}\left[{\mathcal{N}}_T^{1+\gamma'}\left(\dfrac{1}{{\mathcal{N}}_T}\sum_{i=1}^{{\mathcal{N}}_T} \xi_i\right)^{1+\gamma'}\right]\\
&\leq \mathbb{E}\left[ {\mathcal{N}}_T^{\gamma'} \sum_{i=1}^{{\mathcal{N}}_T} \xi_i^{1+\gamma'}\right]\\
&= \mathbb{E}\left[ {\mathcal{N}}_T^{1+\gamma'} \right]\cdot \mathbb{E}[\xi_1^{1+\gamma'}] \leq N^{1+\gamma'} C_{\gamma,T},
\end{align*}
where we used that $N_T$ is the sum of $N$ independent Poisson random variable with parameter depending only on $\gamma$ and $T$.
The extended Markov inequality yields
\[
P_z := \mathbb{P}(X_t \geq z) \leq C_{\gamma,T}\dfrac{N^{1+\gamma'}}{z^{1+\gamma'}}.
\]

Next, we note that, when starting independent random walkers, one at each site left of $-mN$, the number of particles that arrive at sites right of the origin equals the sum of independent Bernoulli trials with parameter $P_{-z}$, $z\leq -mN$.
We conclude that the expected number of particles can be bounded by a constant times 
\[
N^{1+\gamma'}\sum_{z=mN}^{\infty} z^{-1-\gamma'} = N\cdot \dfrac{1}{N}\sum_{z\geq mN} \left(\dfrac{z}{N}\right)^{-1-\gamma'} \lesssim N\int_{m}^{\infty} u^{-1-\gamma'}\;\mathrm{d}u \lesssim Nm^{-\gamma'}.
\]
Together with \eqref{eq:final_equation} and letting $m\to \infty$ after $N\to \infty$, this concludes the proof.

\printbibliography

@preamble{ "\providecommand{\noopsort}[1]{} " }

@Article{BMNS17,
author={Baldasso, Rangel
and Menezes, Ot{\'a}vio
and Neumann, Adriana
and Souza, Rafael R.},
title={Exclusion Process with Slow Boundary},
journal={Journal of Statistical Physics},
year={2017},
month={Jun},
day={01},
volume={167},
number={5},
pages={1112-1142},
issn={1572-9613},
doi={10.1007/s10955-017-1763-5},
}

@Book{DZ98,
author={Dembo, Amir
and Zeitouni, Ofer},
title={Large deviations techniques and applications},
year={1998},
edition={2. ed},
publisher={Springer New York, N.Y.},
address={New York, N.Y.},
isbn={0387984062; 9780387984063},
language={eng}
}

@misc{QT22,
      title={Hydrodynamic large deviations of TASEP}, 
      author={Jeremy Quastel and Li-Cheng Tsai},
      year={2022},
      eprint={2104.04444},
      archivePrefix={arXiv},
      primaryClass={math.PR}
}

@misc{KGX24,
      title={A novel approach to hydrodynamics for long-range generalized exclusion}, 
      author={Patrícia Gonçalves and Julian Kern and Lu Xu},
      year={2024},
      eprint={2410.17899},
      archivePrefix={arXiv},
      primaryClass={math.PR},
      url={https://arxiv.org/abs/2410.17899}, 
}

@article{Sep98,
author = {Seppäläinen, Timo},
year = {1998},
month = {01},
pages = {},
title = {Coupling the totally asymmetric simple exclusion process with a moving interface},
volume = {4},
journal = {Markov Processes and Related Fields}
}

@incollection {Var04,
    AUTHOR = {Varadhan, Srinivasa R. S.},
     TITLE = {Large deviations for the asymmetric simple exclusion process},
 BOOKTITLE = {Stochastic analysis on large scale interacting systems},
    SERIES = {Adv. Stud. Pure Math.},
    VOLUME = {39},
     PAGES = {1--27},
 PUBLISHER = {Math. Soc. Japan, Tokyo},
      YEAR = {2004},
      ISBN = {4-931469-24-8},
   MRCLASS = {60F10 (60K35 82C22 82C31)},
  MRNUMBER = {2073328},
MRREVIEWER = {Boualem\ Djehiche},
       DOI = {10.2969/aspm/03910001},
       URL = {https://doi.org/10.2969/aspm/03910001}
}

@Article{X22b,
author={Xu, Lu},
title={Hydrodynamics for One-Dimensional ASEP in Contact with a Class of Reservoirs},
journal={Journal of Statistical Physics},
year={2022},
month={Jul},
day={21},
volume={189},
number={1},
pages={1},
issn={1572-9613},
doi={10.1007/s10955-022-02963-x},
url={https://doi.org/10.1007/s10955-022-02963-x}
}

@Article{Coc85,
author={Cocozza-Thivent, Christiane},
title={Processus des misanthropes},
journal={Zeitschrift f{\"u}r Wahrscheinlichkeitstheorie und Verwandte Gebiete},
year={1985},
month={Dec},
day={01},
volume={70},
number={4},
pages={509-523},
issn={1432-2064},
doi={10.1007/BF00531864},
url={https://doi.org/10.1007/BF00531864}
}

@misc{Can19,
	author={Clément Canonne},
	year=2019,
	url={http://www.cs.columbia.edu/~ccanonne/files/misc/2017-poissonconcentration.pdf},
	note={URL: \url{http://www.cs.columbia.edu/~ccanonne/files/misc/2017-poissonconcentration.pdf}, Accessed: August 18, 2023}
}

@misc{X22a,
      title={Hydrodynamic limit for asymmetric simple exclusion with accelerated boundaries}, 
      author={Lu Xu},
      year={2022},
      eprint={2108.09345},
      archivePrefix={arXiv},
      primaryClass={math.PR}
}

@unpublished{BCGS22,
  TITLE = {{Hydrodynamic limit for a boundary driven super-diffusive symmetric exclusion}},
  AUTHOR = {Bernardin, C{\'e}dric and Cardoso, Pedro and Gon{\c c}alves, Patr{\'i}cia and Scotta, Stefano},
  URL = {https://hal.science/hal-02884549},
  NOTE = {working paper or preprint},
  YEAR = {2022},
  MONTH = Jun,
  HAL_ID = {hal-02884549},
  HAL_VERSION = {v2}
}

@article{SS18,
	author = {Sunder Sethuraman and Doron Shahar},
	title = {{Hydrodynamic limits for long-range asymmetric interacting particle systems}},
	volume = {23},
	journal = {Electronic Journal of Probability},
	number = {none},
	publisher = {Institute of Mathematical Statistics and Bernoulli Society},
	pages = {1 -- 54},
	year = {2018},
	doi = {10.1214/18-EJP237},
	URL = {https://doi.org/10.1214/18-EJP237}
}

\end{document}